\documentclass{amsart}
\usepackage{amsmath, amssymb}

\newtheorem{Them}{Theorem}
\newtheorem{Cor}{Corollary}
\newtheorem{Prop}{Proposition}
\newtheorem{Remark}{Remark}

\newcommand{\fq}{\mathbb{F}_q}
\newcommand{\rmv}[1]{}

\begin{document}
\title[On coefficients of powers of polynomials and their compositions]{On coefficients of powers of polynomials and their compositions over finite fields}
\author[Mullen]{Gary L. Mullen}
\address{Department of Mathematics,
 The Pennsylvania State University,
 University Park, PA 16802, USA}
\email{mullen@math.psu.edu}

\author[Muratovi\'{c}-Ribi\'{c}]{Amela Muratovi\'{c}-Ribi\'{c}}
\address{University of Sarajevo, Department of Mathematics, Zmaja od Bosne 33-35, 71000 Sarajevo, Bosnia and Herzegovina}
\email{amela@pmf.unsa.ba}

\author[Wang]{Qiang Wang}
\thanks{Research of Qiang Wang is partially supported by NSERC of Canada.}
\address{School of Mathematics and Statistics,
Carleton University, 1125 Colonel By Drive,  Ottawa, Ontario, $K1S$
$5B6$, CANADA}

\email{wang@math.carleton.ca}



\begin{abstract}
For any given polynomial $f$ over the finite field $\fq$ with degree at most $q-1$, we associate it with a 
$q\times q$ matrix $A(f)=(a_{ik})$ consisting of coefficients of its powers
$(f(x))^k=\sum_{i=0}^{q-1}a_{ik} x^i$ modulo $x^q -x$ for $k=0,1,\ldots,q-1$. This matrix has some interesting properties such as $A(g\circ f)=A(f)A(g)$ where $(g\circ f)(x) = g(f(x))$ is the composition of the polynomial $g$ with the polynomial $f$. In particular, $A(f^{(k)})=(A(f))^k$ for any $k$-th composition $f^{(k)}$ of $f$ modulo $x^q-x$ with $k \geq 0$. As a consequence, we prove that  the rank of $A(f)$ gives the cardinality of the value set of $f$.  Moreover,  if $f$ is a permutation polynomial then the matrix associated with its inverse $A(f^{(-1)})=A(f)^{-1}=PA(f)P$ where $P$ is an
antidiagonal permutation matrix.  As an application, we study the period of  a nonlinear congruential pseduorandom sequence $\bar{a} = \{a_0, a_1, a_2, ... \}$ generated by $a_n = f^{(n)}(a_0)$ with initial value $a_0$, in terms of the order of the associated matrix.  Finally we show that $A(f)$ is diagonalizable in some extension field of $\fq$ when $f$ is a permutation polynomial over $\fq$.
\end{abstract}
\maketitle
\section{introduction}
Let $\fq $ be the finite field of order $q=p^n$ where $p$ is a prime number 
and $n$ is a positive integer. Let $f(x) = \sum_{i=0}^{q-1}a_ix^i$ be a
 polynomial over $\fq $ with degree at most $q-1$. To compute  its composition with another polynomial $g(x)=\sum_{i=0}^{q-1}b_ix^i$,  we can either use interpolation to obtain its expression directly, or  calculate all the powers $f(x)^i\pmod {x^q-x}$ in the expression $(g\circ f)(x)=\sum_{i=0}^{q-1}b_i(f(x))^i$. 

\par Denote by
$$(f(x))^k=\sum_{i=0}^{q-1}a_{ik}x^i ~\mod{(x^{q} -x)}$$
the $k$-th power of the polynomial $f(x)$ for $k=1,2,\ldots ,q-1$. Denote by $f_0$ the zero polynomial in $\fq[x]$. If $f\neq f_0$ we will define $(f(x))^0=1$ and $f_0(x)^0=0$.
\par For any polynomial $f(x)=\sum_{i=0}^{q-1}a_ix^i$ we associate a coefficient vector $v_f$ with it, namely,  
$$v_f=(a_0,a_1,\ldots ,a_{q-1})^T.$$
We define a $q\times
q$ matrix associated with $f(x)\neq f_0(x)$ by
$$A(f)=\begin{bmatrix}  1& a_{01}   &a_{02} &\ldots &a_{0,q-2}&a_{0,q-1}\\
0& a_{11} &a_{12} &\ldots &a_{1,q-2}&a_{1,q-1}\\
\vdots &\vdots &\ddots &\vdots &\vdots\\
0& a_{q-2, 1} &a_{q-2,2} &\ldots &a_{q-2,q-2}&a_{q-2,q-1}\\
0& a_{q-1,1} &a_{q-1,2} &\ldots &a_{q-1,q-2}&a_{q-1,q-1}\end{bmatrix},$$
where the $k$-th column consists of the coefficients of the $(k-1)$-th power of $f(x)$. 
In particular, we define $A(f_0)$ to be the zero $q\times q$ matrix. We note that we can build the matrix $A(f)$ by directly computing each of the $k$-th powers of $f(x)$ modulo $x^q-x$. For example, finding each column of $A(f)$ takes $q^{1+o(1)}$ 
bit operations  using the result of Kedlaya and Umans \cite{KU}.  On the other hand, using  Lagrange's interpolation $f(x)^k = \sum_{\alpha \in \fq} f(\alpha)^k \left(1-(x-\alpha)^{q-1}\right)$,   one can obtain the explicit expression for all the entries of $A(f)$. Namely, for all $1\leq i, j \leq q-1$, we have $a_{ij} = - \sum_{\alpha \in \fq} f(\alpha)^j {q-1 \choose i } (-\alpha)^{q-1-i}$ and $a_{0j} = f(0)^j = \sum_{\alpha\in \fq} f(\alpha)^j (1- (-\alpha)^{q-1})$.

The {\it Bell matrix} of an analytic function $f$ is an infinite matrix defined as
\[
B[f]_{jk} = \frac{1}{j!}\left[\frac{d^j}{dx^j} (f(x))^k \right]_{x=0},
\]
where  $(f(x))^k = \sum_{j=0}^{\infty} B[f]_{jk} x^j$. It is sometimes called a {\it 
Jabotinsky matrix}. The transpose of a Bell matrix is called a {\it Carleman matrix}, which is often used in iteration theory to find the continuous iteration of a function \cite{GK:02}.

In this paper we show that our matrix $A(f)$ of a polynomial $f$ over $\fq$ is indeed a finite field analogue of the Bell matrix. Some fundamental properties in terms of the composition of polynomials are proved similarly. Moreover, we derive a few results specifically related to finite field theory.   
In Section~\ref{section:general} we show that the matrix associated with the composition of two polynomials over a finite field is the product of two associated matrices. That is, $A(g\circ f) = A(f) A(g)$.   As a corollary, we prove that the value set size of any polynomial $f$ over $\fq$ is the rank of its associated matrix $A(f)$, which is equivalent to an earlier result of Chou and Mullen \cite{ChouMullen}, which deals with the transpose of the $(1,1)$-minor of $A(f)$.  In Section~\ref{section:PP}, we concentrate on permutation polynomials over $\fq$. In particular, we prove that the associated matrix for the compositional inverse $f^{(-1)}$ satisfies $A(f^{(-1)}) = PA(f)P$, where $P$ is an antidiagonal permutation matrix defined by $P_{i, q-i} =1$ for $i=1, \ldots, q$ and zero otherwise. Moreover, we show $A(f)$ is  diagonalizable in some extension field of $\fq$. 
Throughout this paper, we note that $f^k(x)$ or $(f(x))^k$ denotes the $k$-th power of $f(x)$ modulo $x^q-x$, while $f^{(k)}(x)$ denotes the $k$-th composition of $f(x)$ modulo $x^q-x$.  

\section{The matrix of a composition of polynomials} \label{section:general}

First we derive the following obvious result.
\begin{Prop} \label{compo} Let $f(x)=\sum_{i=0}^{q-1}a_ix^i\in \fq [x]$ and $g(x)=\sum_{i=0}^{q-1}b_ix^i\in \fq [x]$. Then
$$v_{g\circ f}=A(f)v_g.$$
\end{Prop}
\begin{proof} The $(i+1)$-th coordinate of $A(f)v_g$ is given by
$(A(f)v_g)_{i+1}=\sum_{k=0}^{q-1}a_{ik} b_k$. On the other hand, we obtain $g\circ f(x)=\sum_{k=0}^{q-1}b_k(f(x))^k=\sum_{k=0}^{q-1}b_k\sum_{i=0}^{q-1}a_{ik}x^i=\sum_{i=0}^{q-1}(\sum_{k=0}^{q-1}b_ka_{ik})x^i$. Therefore we obtain $(v_{g\circ f})_{i+1}=\sum_{k=0}^{q-1}a_{ik} b_k=(A(f)v_g)_{i+1}$ for $i=0,1,\ldots q-1$.\end{proof}

\begin{Them} Let $g(x)=\sum_{i=0}^{q-1}c_ix^i \in \fq [x]$ and $f(x)=\sum_{j=0}^{q-1}a_j x^j\in \fq [x]$. Let $(g\circ f)(x) = g(f(x))$ be the composition of $g$ with $f$.  Then
$$A(g\circ f)=A(f)A(g).$$
\end{Them}
\begin{proof}
By Proposition~\ref{compo}, we see that $A(f)v_{g^k}=v_{g^k\circ f}$ for any $k$-th power of the polynomial $g$. Let $\sigma_k(x)=x^k$.  Because the composition of polynomials is an associative operation, we have $g^k\circ f=(\sigma_k \circ g)\circ f=\sigma_k \circ (g\circ f)=(g\circ f)^k$. Therefore $A(f)v_{g^k}=v_{(g\circ f)^k}$ for all $k=0,1,2\ldots ,q-1$. Partitioning the matrix $A(g)$ with columns $v_{g^0},v_{g},v_{g^2},\ldots ,v_{g^{q-1}}$, we derive
$$A(f)A(g)=\Big(A(f)v_{g^0},A(f)v_{g},A(f)v_{g^2},\ldots ,A(f)v_{g^{q-1}}\Big)$$
$$=\Big(v_{(g\circ f)^0},v_{(g\circ f)},v_{(g\circ f)^2},\ldots ,v_{(g\circ f)^{q-1}}\Big)=A(g\circ f).$$
\end{proof}

We recall that $f^k(x)$ denotes the $k$-th power of $f(x)$, while $f^{(k)}(x)$ denotes the $k$-th composition of $f(x)$.  

\begin{Cor} For any given polynomial $f\in \fq[x]$ we have that $A(f^{(k)})=(A(f))^k$, for any $k=1,2,\ldots $.
\end{Cor}

This provides an algebraic way to study the composition of polynomials in terms of multiplication of matrices. Although the matrices associated with polynomials are large and costly to build, this still gives us some interesting theoretical consequences. We note that the transpose of the $(1, 1)$-minor of $A(f)$ was earlier studied by  Chou and Mullen \cite{ChouMullen}. They gave a result on the size of the value set of $f$ in terms of the rank of the $(1, 1)$-minor of $A(f)$;  see also page 234 in \cite{Handbook}. However, our proof is different. 

\begin{Cor} \label{ChouMullen}
Let $f$ be a polynomial over a finite field $\fq$ and  $|V_f|$ be the size of the value set $V_f=\{f(a) \mid a \in \fq\}$ of $f$. Then $|V_f| = rank (A(f))$. 
\end{Cor}
\begin{proof}
If $f(x)\in \fq [x]$ is not a permutation polynomial then we define $D=V_f$ and let $g\in \mathbb{F}_q [x]$ be a nonzero polynomial of least degree $m$ such that $g:D\rightarrow \{0\}$. 
Let $g(x)=b_mx^m+b_{m-1}x^{m-1}+\cdots +b_1x+b_0$. Then we have $g\circ f(x)= 0$, and thus $A(f)v_g=0$ by Proposition~\ref{compo}.   This means that the first $m+1$ columns of $A(f)$ are linearly dependent and thus the coefficients of $g$ determine a linear dependence among the  polynomials $1,f(x),f^2(x),\ldots ,f^{m}(x)$ in the sense that $\sum_{i=0}^{m}b_i(f(x))^i=0$. 
Moreover, $(f(x))^0,f(x),\ldots ,(f(x))^{m-1}$ are linearly independent because $g(x)$ is the lowest degree polynomial such that $g\circ f =0$. Therefore,  $rank(A(f))=deg(g(x))=|V_f|.$ 

If $f\in \fq [x]$ is permutation polynomial, then all the powers of $f$ and corresponding columns of $A(f)$ are linearly independent.
\end{proof}

Corollary~\ref{ChouMullen} states that the size of the value set of $f$ is given by the rank of the matrix $A(f)$. One would also like to know which elements $c \in \fq$ show up in the value set $V_f$ of $f$, and if $c$ shows up in the value set, how many times does it appear?

Again, we consider the polynomial  $f(x) = a_0 +a_1 x + \cdots + a_{q-1} x^{q-1}$ over $\fq$. First we want to determine the number of nonzero solutions to $f(x)=c$.  Let us consider the polynomial $h(x) =
(a_0 +a_{q-1} -c) +a_1 x + \cdots + a_{q-2}x^{q-2}$. Then, by the K\"{o}nig-Rados Theorem (Theorem 6.1 in \cite{LN:97}), the number of nonzero solutions to $f(x) = c$ is
$q-1-r$, where $r$ the rank of the $(q-1)\times (q-1)$ left circulant matrix 
$$C(h):= \begin{bmatrix}  
a_0 + a_{q-1}-c & a_{1}  &\ldots &a_{q-2}\\
a_1& a_2 &\ldots &a_0 + a_{q-1}-c\\
a_2 & a_3 & \ldots & a_1 \\
\vdots &\vdots &\vdots &\vdots\\
a_{q-3}& a_{q-2} &\ldots &a_{q-4}\\
a_{q-2} & a_{0}+ a_{q-1}-c  &\ldots & a_{q-3}\end{bmatrix}.$$

Therefore, if $c \neq f(0)$, then $c$ appears in the value set $V_f$ of $f$ if and only if the rank of  the matrix $C(h)$ is less than $q-1$. And the number of times that $c$ appears in the value set $V_f$ of $f$ is $r$ if and only if  the rank of the matrix $C(h)$ is $q-1-r$. If $c = f(0)$ then $c$ appears in the value set $q-r$ times.

Let $k$ be the largest positive integer such that $\{1, f, \ldots, f^{k}\}$ is linearly independent over $\fq$. Then obviously the rank $rank(A(f)) \geq k+1$. For example, let $f\in \fq[x]$ be a polynomial of degree $d$, then it is obvious that $1, f, \ldots, f^{\lfloor (q-1)/d \rfloor}$ are linearly independent and thus the value set $V_f$ has size $|V_f| \geq \lfloor (q-1)/d \rfloor +1$. We note that polynomials satisfying $|V_f| = \lfloor (q-1)/d \rfloor +1$ are called {\it minimum value set polynomials}. The classification of minimum value set polynomials is the subject of several papers; see \cite{BC,Carlitzetal:61,Gomez:88,GomezMadden:88,Mills:64}. Using the discussion after Corollary~\ref{ChouMullen}, we have the following.

\begin{Cor}
Let $f$ be a polynomial of degree $d$ over the finite field $\fq$. Then $f$ is a minimum value set polynomial if and only if $rank(A(f)) = \lfloor (q-1)/d \rfloor +1$. That is equivalent to, $\{1, f, \ldots, f^{\lfloor (q-1)/d \rfloor }\}$ is a basis which spans the space of all nonnegative powers of $f$. 
\end{Cor}

Let us consider the $(1,1)$-minor $M(f)$ of $A(f)$.  If
the $i$-th row of $M(f)$ consists entirely of 0's or entirely of 1's, set $l_i$ = 0.
Otherwise for a nonzero $i$-th row of $M(f)$, arrange the entries in a circle and
define $l_i$ to be the maximum number of consecutive zeros appearing in this
circular arrangement. Let $L_f$ be the maximum of the values of $l_i$, where the
maximum is taken over all of the $q-1$ rows  of the matrix $M(f)$. Using the linearly independence of these columns, we can derive a lower bound of the size of the value set $V_f$. 

\begin{Cor} (Theorem 1, \cite{DasMullen})
Let $f$ be a polynomial over $\fq$ and $L_f$ be defined as above. Then 
$|V_f| \geq L_f +2$. 
\end{Cor}

In \cite{Mullen}, Mullen fully classified polynomials $f(x)$ over finite fields which commute with linear permutations, that is, $f(bx+a) = bf(x) +a$. We note that $A(bx+a)$ is an upper triangular matrix.  Comparing the second column of $A(bx+a) A(f) = A(f) A(bx+a)$, one can derive the following corollary. 

\begin{Cor} (Theorem 1,  \cite{Mullen})
The polynomial $f(x) = b_0 + b_1 x + \cdots + b_{q-1} x^{q-1}$ satisfies $f(bx+a) = bf(x) +a$ if and only if 
\begin{eqnarray}
b_0 (b-1) &=& -a + \sum_{t=1}^{q-1} b_t a^t, \\
b_s (1-b^{s-1}) &=& b^{s-1}\sum_{t=s+1}^{q-1} {t\choose s}a^{t-s}b_t,  ~~~(1\leq s \leq q-1) \nonumber
\end{eqnarray}
\end{Cor}

\rmv{
\begin{Remark}
Let $f(x)\in \fq[x]$. It is obvious that $\lambda =1$ is an eigenvalue of $A(f)$. $e_1$ is one of its eigenvectors. Let $v\in \fq ^n$ be an arbitrary vector with coordinates $v_i$. If $v\neq 0$ we can associate a polynomial $g$ to vector $v$ so that $v=v_g$. Then $e_1=f_0$. If $v_g$ is eigenvector of $A(f)$ corresponding to eigenvalue $\lambda =1 $ then $g\circ f(x)=g(x)$ for all $x\in \fq$. Set $\fq $ can be partitioned to the least subsets $S_i$ such that $f:S_i\rightarrow S_i$ and $\cup_{i=1}^m S_i=\fq $ ({\bf I don't understand this sentence}). Defining polynomial $g_i(x)=1$ if $x\in S_i$ and zero otherwise, for $i=1,2,\ldots ,m$, we can see that $g_i\circ f(x)=g_i$, vectors $v_{g_i}$ are linearly independent and together with $e_0$ generates the  eigenspace $E_1$ i.e. there is a vector subspace of polynomials such that $g\circ f=g$. For any $\mu \in \fq $, let $\delta_{\mu }(x)=\mu x\in \fq [x]$. Assume $g\circ f(x)=\mu g(x)$.  Then if permutation polynomial $g(x)$ induces eigenvector of $A(f)$ it means that $f(x)$ is also permutation polynomial conjugate with $\delta_{\mu }$.
\end{Remark}
}

\section{Permutation polynomials}\label{section:PP}
Permutation polynomials over the field $\mathbb{F}_q$ under the operation of functional composition form a   group isomorphic to the symmetric group $(S_q,\circ )$ with $q!$ elements. There is a representation of the permutation polynomials in terms of circulant matrices such that its centralizer is commutative \cite{Amela0},  but here we consider the representation of $f$ in terms of the invertible matrix $A(f)$. We note that the mapping $f\rightarrow A(f)$ is one-to-one. Hermite's criterion (Theorem 7.4 in \cite{LN:97}) states that $f(x)$ is
permutation polynomial if and only if the coefficient $a_{q-1, k}$ in the $k$-th power of $f(x)$ is $0$  for all $k=1,2,\ldots ,q-2$ and $f(x)$ has exactly one root in $\fq$, say $f(e) =0$.  This means that all entries of the last row of $A(f)$ are zero except $a_{q-1, q-1} =1$. Indeed,  $f(x)^{q-1} = 1$ if $x \neq e$ and $f(x)^{q-1} =0$ if $x=e$. Hence 
\[
f(x)^{q-1} = \sum_{a\neq e} (1-(x-a)^{q-1}) = x^{q-1} - \sum_{i=1}^{q-2}  \sum_{a \neq e} {q-1 \choose i} (-a)^{q-1-i}  x^i + \sum_{a \neq e} (1- (-a)^{q-1}). 
\]
Therefore $a_{q-1, q-1} = 1$. Moreover,  $a_{0, q-1} = 1$ if $e \neq 0$, and 
$a_{0, q-1} = 0$ if $e = 0$.

We now consider the compositional inverse $f^{(-1)}$ of a permutation polynomial $f$ with respect to composition. Since $A(g\circ f)=A(f)A(g)$ and the matrix associated with $f^{(0)}(x)=Id(x)=x$ is the identity matrix, it is easy to see that $A(f^{(-1)}) = A(f)^{-1}$. Moreover, we have


\begin{Them}\label{inverse} Let $f$ be a permutation polynomial of $\fq $. Let $P$ be the  antidiagonal permutation matrix, i.e. $P$ is defined by $P_{i,(q-i)}=1$ for $i=1,2,\ldots, q$ and zero otherwise. Then $A(f^{(-1)})=(A(f))^{-1} = PA(f)P$.
\end{Them}
\begin{proof}

Obviously, $A(f^{(-1)})=(A(f))^{-1}$.  Denote the $k$-th power of $f$ and the inverse polynomial $f^{(-1)}$ by 
$f^{k}(x)=\sum_{i=0}^{q-2}a_{ik}x^i$ and $(f^{(-1)}(x))^k=\sum_{i=0}^{q-2}b_{ik}x^i$ respectively, for $k=1,2,\ldots, q-1$. For any permutation polynomial $g(x)=\sum_{i=0}^{q-2}c_ix^i$,  it is well known (see for example \cite{Amela}) that its coefficients can be
calculated by $c_i=-\sum_{s\in \fq }s^{q-1-i}g(s),$ for $i=0, 1,\ldots ,q-2$. 

For $1\leq k \leq q-2$, by Hermite's criterion, the polynomial $f^k(x)$ must have degree at most $q-2$. 
Therefore we have for $0\leq i \leq q-2$ and $1\leq k \leq q-2$, 
$$a_{ik}=-\sum_{s\in \fq }(f(s))^ks^{q-1-i} = -\sum_{s\in \fq}s^k(f^{(-1)}(s))^{q-1-i})=b_{q-1-k, q-1-i},$$
i.e.
$$a_{ik}=b_{q-1-k, q-1-i},  ~for ~0\leq i \leq q-2~ and~ 1\leq k \leq q-2. $$
Moreover, $a_{q-1, k} =0$ for $ 1\leq k \leq q-2$ and $a_{q-1, q-1} =1$. In addition, $a_{q-1, k} = b_{q-1-k, 0}$ by the definition of $A(f^{(-1)})$.

On the other hand, for   any polynomial $g(x)=\sum_{i=0}^{q-1}c_ix^i$,  it is well known that its coefficients can be
calculated by $c_i=-\sum_{s\in \fq }s^{q-1-i}g(s),$ for $i=1,\ldots ,q-2$ and 
$c_0 + c_{q-1} = -\sum_{s\in \fq } g(s)^{q-1}$. Hence we can compute

$$a_{i, q-1} = - \sum_{s\in \fq } s^{q-1-i} (f(s))^{q-1} = -\sum_{s\in \fq} (f^{(-1)}(x))^{q-1-i} s^{q-1}  = b_{0, q-1-i},$$ 
for $1\leq i \leq q-2$ and $a_{0, q-1} + a_{q-1, q-1} = -\sum_{s\in \fq} f(s)^{q-1} = -\sum_{s\in \fq} s^{q-1} = 1$. Because $a_{q-1, q-1} =1$, we have
$a_{0, q-1} =0$, which is equal to $b_{0,q-1}$ by definition.  Hence we have proven that $a_{ik}=b_{q-1-k, q-1-i}$, for all $0\leq i \leq q-1$ and $1\leq k \leq q-1$. Since the multiplication by $P$ on both sides reverses the order of rows and columns of $A(f)$, it follows that $A(f^{(-1)})=PA(f)P$.
\end{proof}
\begin{Cor} Let $f$ be a permutation polynomial and $P$ be the antidiagonal permutation matrix as defined in Theorem~\ref{inverse}. Then the matrix $PA(f)$ is the inverse of itself.
\end{Cor}
\begin{proof}  By Theorem~\ref{inverse}, we have $(A(f))^{-1}=A(f^{(-1)})=PA(f)P$. Therefore $(P(A(f))^2=I$.
\end{proof}
\begin{Cor} Let $S$ be a group of invertible $q\times q$ matrixes over $\fq $ equipped with the operation $A *B=B\cdot A$ where $B\cdot A$ denotes the usual product of the matrices $B$ and $A$. Denote by $f_{\pi }$ the permutation polynomial of degree at most $q-2$ induced by a permutation $\pi \in S_q$. Then the mapping $\psi :S_q\rightarrow  S$ given by $\psi (\pi )=A(f_{\pi })$ is a monomorphism and thus $S_q$ is isomorphic to the  subgroup $\mathcal{A}=\{A(f_{\pi})|\pi \in S_q\}$ of the group $S$.
\end{Cor}
\begin{proof}
It is easy to show that $(S, *)$ is a group and that the mapping $\psi $ is injective. Now $\psi (\pi \circ \alpha )=A(f_{\pi }\circ f_{\alpha })=A(f_{\alpha })\cdot A(f_{\pi })=\psi (\pi )*\psi (\alpha )$.
\end{proof}

Finally we comment on some potential applications of our results in sequence designs.  For any permutation polynomial $f$,  we define a nonlinear congruential pseudorandom sequence $\bar{a} = \{a_0, a_1, a_2, ... \}$ such that $a_n = f^{(n)}(a_0)$ and $a_0$ is the initial value. The period of $\bar{a}$ is equal  to  the smallest $k$ such that  $f^{(k+i)} (a_0) = f^{(i)}(a_0)$ for some $i$. Character sums of these sequences are studied in \cite{Domigo, Nied1, Nied2, Nied3, Nied5}.
For each initial value that is not fixed by $f$, we find the period of the  nonlinear congruential pseudorandom sequence. If we take $K$ as the least common multiple of all these periods, then we obtain $f^{(K)} = id$ and thus $A(f^{(K)}) = I$.   Conversely,  if $A(f)^K = I$ then the period of the nonlinear congruential sequence is a divisor of the order of the invertible matrix.  Next we demonstrate  the following two simple examples although they can be obtained easily without using these matrices. 

Let $f(x) = x^m$ be a polynomial over $\fq$ such that $(m, q-1) =1$. Then $A(f)$ is a permutation array such that the only nonzero entry in column $k$ is in $(km \pmod{q-1}, k)$ position where $1 \leq k \leq q-1$. The period of $\bar{a}$ is well known, which is  the order of $m$ modulo $q-1$.

Let $f(x) = ax+b \in \mathbb{F}_p[x]$, where $a$ is a primitive element in $\mathbb{F}_p$ and $b\neq 0$. Then
$$A(f)=\begin{bmatrix}  1& b   &b^2 &\ldots &b^{p-2}&1\\
0& a &2ab &\ldots &(p-2)ab^{p-3}& (p-1) a b^{p-2}\\
\vdots &\vdots &\ddots &\vdots &\vdots\\
0& 0 &0 &\ldots &a^{p-2}& (p-1) a_{p-2}b\\
0& 0 & 0 &\ldots & 0 & 1\end{bmatrix}.$$

The matrix is an upper triangular matrix such that its eigenvalues are all the nonzero elements ($a^k$, $k=1, \ldots, p-1$) in $\mathbb{F}_p$ and the multiplicity of $1$ is $2$. Hence the period of $A$ is equal to $p-1$,  the least common multiple of orders of these eigenvalues.

Computing the order of the matrix $A(f)$ associated with a permutation polynomial $f$ provides an algebraic way to find out the period of this kind of pseudorandom sequence, although the matrix $A(f)$ itself is costly to build.  For example, finding each column of $A(f)$ takes $q^{1+o(1)}$ 
bit operations  using the result of Kedlaya and Umans \cite{KU}. We wonder whether we could overcome this drawback by pre-computing the initial matrix, or taking a sparse matrix, or diagonalizing the matrix. 
As an attempt, we end our paper with a diagonalization result of $A(f)$ over some extension field of $\fq$. 

\begin{Them} \label{eigenvalue}
Let $f \in \fq[x]$ be a permutation polynomial of $\fq$ such that the disjoint cycles  $C_1,C_2,\ldots, C_k$ of $f$ have lengths $L_1, L_2, \ldots, L_k$ respectively. Let $K$ be an extension field of $\fq$ that contains all solutions of the equations $x^{L_i}-1=0$ for $i=1,2,\ldots,k$ and $\psi_i$ be a fixed primitive $L_i$-th root of unity in $K$ for each $i$. Then  $A(f)$ is diagonalizable with eigenvalues $\psi_{i}^{j}$ for $i=1, \ldots, k$ and $j=0, \ldots, L_i -1$.
\end{Them}

\begin{proof}
 For each cycle $C_i$ we pick a starting point (arbitrarily) and denote it by $b_0$, so our cycle can be denoted by $(b_0,b_1,\ldots ,b_{L_i-1})$.
For each $j$ such that $0\leq j \leq L_i-1$,  we can  define
\begin{eqnarray*}
g_{i,j}(x)  &=&\left\{ 
\begin{array}{ll}  (\psi_i^j)^t &  ~if~ x = b_t \in C_i; \\
0 & ~if~ x\in \fq \setminus C_i. 
\end{array}\right.
\end{eqnarray*}

Obviously,  
$$g_{i,j}(f(x))=(\psi_i^j)g_{i,j}(x)$$
i.e., each $g_{i, j}$ produces an eigenvector of $A(f)$ with the corresponding eigenvalue $(\psi_i^j)$.
Indeed, if $x\not \in C_i$,  then $f(x)\not \in C_i$ and so $g_{i,j}(x)=0=g_{i,j}(f(x))$. If $x\in C_i$ then $x=b_t$ for some $t=0,1,\ldots ,L_i-1$. Then $f(x)=b_{t+1\pmod{ L_i}}$. Thus $g_{i,j}(f(x))= (\psi_i^j)^{t+1}=(\psi_i^j) (\psi_i^j)^t=(\psi_i^j)g_{i,j}(x)$. In this way we obtain a set $\{g_{i,j}(x):i =1, \ldots, k, j=0,1,\ldots,L_i-1\}$ of $q$ polynomials in $K[x]$. For each fixed $i$, it is easy to see that  $\{g_{i, j}(x): j =0, \ldots, L_i-1\}$ is linearly independent because $\psi_i$ is a primitive  $L_i$-th root of unity. Moreover, if $i\neq i^\prime$ then $g_{i, j}(x) g_{i^\prime, j^\prime} (x) = 0$. Therefore the set of $q$ polynomials $\{g_{i,j}(x): i =1, \ldots, k, j=0,1,\ldots,L_i-1\}$ is linearly independent. Because the size of $A(f)$ is $q$ and these $g_{i, j}(x)$'s provide us $q$ linearly independent eigenvectors corresponding to eigenvalue $\psi_i^j$, the proof is complete.
\end{proof}


\begin{Remark} From the proof of Theorem~\ref{eigenvalue}, all polynomials $g(x) \in K[x]$ such that $g(f(x)) = \lambda g(x)$ for some $\lambda$ satisfy
$$g(x)=\sum_{i=1}^k \sum_{j=0}^{L_i-1}a_{i,j}g_{i,j}(x). $$
\end{Remark}

\begin{Remark}
 Theorem~\ref{eigenvalue} can be extended to non permutation polynomials such that either $x$ or $f(x)$ is in a cycle of the functional graph of $f$,  that is, the tail length of any element in the functional graph is at most one.
 For each such a leaf $d$ in the functional graph of $f$, we define the function
\begin{eqnarray*}
g_{i,d}(x)  &=&\left\{
\begin{array}{ll}  1 &  ~if~ x = d; \\
0 & ~if~ x \neq d.
\end{array}\right.
\end{eqnarray*}
Obviously,  $d\not\in V_f$.  Hence $g_{i,d}(f(x))=0=0g_{i,d}(x)$ for all $x\in \fq$ and thus $g_{i,d}(x)$ derives an eigenvector corresponding to  the eigenvalue $0$. Together with the eigenvectors corresponding to the nodes in the cycles, we have $q$ linearly independent eigenvectors and thus $A(f)$ is diagonalizable. However, in general $A(f)$ is not necessarily diagonalizable in each of its extension fields.  For example, let $f(x) = x^2 + x + 1 \in \mathbb{F}_{5}[x]$. Then
\[
A(f) = \left(
\begin{array}{ccccc}
1 & 1 & 1 & 1 & 1\\
0 & 1 & 2 & 1 & 0\\
0 & 1 & 3 & 2 & 0\\
0 & 0 & 2 & 2 & 0\\
0 & 0 & 1 & 1 & 0
\end{array}
\right)
\]
It is easy to check that the rank of $A(f)$ is $3$ over $\mathbb{F}_5$.  However, eigenvalues of $A(f)$ over 
$\mathbb{R}$ are $5, 1, 1, 0, 0$ and thus are $0, 1, 1, 0, 0$  over $\mathbb{F}_5$. Hence $A(f)$ can not be diagonalizable over any extension field of $\mathbb{F}_5$. 
\end{Remark}

\end{document}